\documentclass[reqno]{amsart}
\usepackage{amsmath,amssymb,amsthm,amsfonts,mathtools}
\usepackage[a4paper,margin=1in]{geometry}
\usepackage{hyperref}
\usepackage[dvipsnames]{xcolor}
\hypersetup{
  colorlinks=true,
  linkcolor=MidnightBlue,
  citecolor=MidnightBlue,
  urlcolor=MidnightBlue
}
\usepackage{enumitem}
\usepackage{mathrsfs}
\usepackage{bm}

\allowdisplaybreaks
\newtheorem{theorem}{Theorem}[section]
\newtheorem{remark}[theorem]{Remark}
\newtheorem{lemma}[theorem]{Lemma}
\newtheorem{proposition}[theorem]{Proposition}
\newtheorem{corollary}[theorem]{Corollary}
\newtheorem{definition}[theorem]{Definition}
\newtheorem{example}[theorem]{Example}

\numberwithin{equation}{section}

\newcommand{\R}{\mathbb{R}}
\newcommand{\N}{\mathbb{N}}
\newcommand{\Om}{\Omega}

\newcommand{\dd}{\,\mathrm{d}x}

\newcommand{\norm}[1]{\left\lVert #1 \right\rVert}
\newcommand{\abs}[1]{\left| #1 \right|}
\newcommand{\ip}[2]{\left\langle #1,#2\right\rangle}
\newcommand{\mcH}{\mathcal{H}}
\newcommand{\mcM}{\mathcal{M}}
\newcommand{\mcS}{\mathcal{S}}
\newcommand{\varrh}{\varrho_{\mcH}}

\title[On Polyharmonic Kirchhoff double phase problems]{On polyharmonic
Kirchhoff double phase problems without AR-conditions}

\author[A. Dixit]{Ashutosh Dixit}
\author[T. Mukherjee]{Tuhina Mukherjee}

\thanks{
Ashutosh Dixit, Department of Mathematics, Indian Institute of Technology Jodhpur,
Rajasthan 342030, India. Email: \href{mailto:p23ma0014@iitj.ac.in}{p23ma0014@iitj.ac.in}.
}

\thanks{
Tuhina Mukherjee, Department of Mathematics, Indian Institute of Technology Jodhpur,
Rajasthan 342030, India. Email: \href{mailto:tuhina@iitj.ac.in}{tuhina@iitj.ac.in}.
}

\thanks{
Corresponding author: Ashutosh Dixit.
}

\subjclass[2020]{35A15, 35J35, 35G20, 35J60}

\keywords{Polyharmonic double phase operator, Kirchhoff problems, Musielak--Orlicz--Sobolev spaces, mountain pass theorem, symmetric mountain pass theorem}

\begin{document}

\begin{abstract}
In this paper, we study a class of polyharmonic Kirchhoff problems driven by a double phase operator. The reaction term has subcritical growth but does not satisfy the Ambrosetti--Rabinowitz condition. Motivated by the work of Harrabi-Hamdani-Fiscella \cite{Harrabi-Hamdani-Fiscella-2024} on m-polyharmonic Kirchhoff problems without Ambrosetti--Rabinowitz conditions, we extend their analysis to a nonhomogeneous double phase setting. We study the problem in the natural Musielak--Orlicz--Sobolev framework associated with the double phase structure. The main novelty of the paper lies in combining the nonlocal Kirchhoff term with a higher-order double phase operator under assumptions weaker than the classical Ambrosetti--Rabinowitz condition. By developing suitable modular estimates and compactness arguments, we establish the variational setting and obtain existence and multiplicity results via minimax methods.

\end{abstract}

\maketitle

\section{Introduction}

Nonlinear higher-order elliptic equations arise naturally in several models from
mathematical physics, geometry, and continuum mechanics. A classical example is
the biharmonic equation, which describes the bending of thin elastic plates in the
Kirchhoff--Love theory \cite{Love-1927}. Higher-order operators also appear in
strain-gradient elasticity and in elastic media with microstructure, where additional
derivatives are needed to describe internal stresses and microscopic deformations
\cite{Antman-2005,Mindlin-1964}. Further applications occur in low Reynolds number
hydrodynamics, structural engineering and nonlinear optics
\cite{Meleshko-2003,Selvadurai-2000,Fibich-Ilan-Papanicolaou-2002}. Moreover,
higher-order equations are closely related to phase separation models, such as the
Cahn--Hilliard equation \cite{Cahn-Hilliard-1958}, and to recent models involving
complex fluids and soft matter \cite{Malchiodi-Mandel-Rizzi-2018,Rizzi-2017}.

From the analytical point of view, the passage from second-order elliptic equations
to polyharmonic equations is far from formal. Indeed, many tools which are available
for equations driven by the Laplacian do not extend naturally to higher-order
operators. For instance, maximum principles and P\'olya--Szeg\H{o}-type inequalities
are generally not available for biharmonic and polyharmonic problems. Moreover, the
Green function associated with higher-order operators may change sign even in simple
domains, which makes positivity arguments and representation formulas substantially
more delicate. In addition, higher-order Sobolev embeddings are less favorable, and
the treatment of nonlinear terms requires refined compactness and convergence
arguments.

The study of nonlinear polyharmonic equations has therefore developed as an
important branch of nonlinear elliptic theory. We refer to the monograph of
Gazzola--Grunau--Sweers \cite{Gazzola-Grunau-Sweers-2010} for a systematic account
of polyharmonic boundary value problems and positivity preserving properties for
higher-order elliptic equations. A general variational identity for nonlinear
elliptic problems was obtained by Pucci--Serrin \cite{Pucci-Serrin-1986}. In the
whole space setting, Berestycki--Lions \cite{Berestycki-Lions-1983} developed a
fundamental variational approach for nonlinear scalar field equations, which has
inspired several later works on higher-order analogues. More recently, Mederski
\cite{Mederski-2020,Mederski-2021} obtained nonradial solutions and developed a
general class of optimal Sobolev inequalities connected with nonlinear scalar field
equations. In the biharmonic case, Mederski--Siemianowski
\cite{Mederski-Siemianowski-2023} studied nonlinear scalar field equations and
established the existence of ground state solutions. Very recently,
Cannone--Cingolani--Mederski \cite{Cannone-Cingolani-Mederski-2025} extended this
line of research to genuine polyharmonic equations of the form
\[
(-\Delta)^m u=g(u)\quad \text{in }\mathbb R^N,
\]
obtaining ground state solutions under general subcritical assumptions inspired by
Berestycki--Lions \cite{Berestycki-Lions-1983} and proving a new polyharmonic logarithmic Sobolev inequality.

In recent years, problems with nonstandard growth have received considerable
attention, especially those driven by double phase operators. Such problems are
naturally associated with variational integrals of the form
\begin{equation}\label{eq:intro-double-phase}
    u \longmapsto
    \int_{\Omega}\left(|\nabla u|^p+a(x)|\nabla u|^q\right)\,dx,
    \qquad 1<p<q,
\end{equation}
or, more generally, with variable exponent densities involving both $p(x)$- and
$q(x)$-growth. The terminology ``double phase'' reflects the fact that the
ellipticity of the integrand changes according to the behavior of the modulating
coefficient $a(\cdot)$. In the region where $a(x)>0$, the leading growth is of
$q$-type, while in the region where $a(x)=0$, the ellipticity is of $p$-type.
This mixture of different growth regimes makes the associated operator
nonhomogeneous and leads to a number of delicate analytical phenomena.

The functional \eqref{eq:intro-double-phase} was introduced by Zhikov
\cite{Zhikov-1986,Zhikov-1995,Zhikov-2011} in the study of strongly anisotropic
materials and homogenization problems. In elasticity theory, the coefficient
$a(\cdot)$ describes the geometry of composites made of two different materials
with distinct hardening exponents. Double phase operators also appear in several
models from physics and engineering, including transonic flow problems, quantum
physics and reaction--diffusion systems; see
\cite{Bahrouni-Radulescu-Repovs-2019,Benci-Davenia-Fortunato-Pisani-2000,
Cherfils-Ilyasov-2005}.

The regularity theory for double-phase functionals has been developed in a
large body of work. We refer, for example, to
Baroni--Colombo--Mingione \cite{Baroni-Colombo-Mingione-2015,
Baroni-Colombo-Mingione-2018}, Baroni--Kuusi--Mingione
\cite{Baroni-Kuusi-Mingione-2014}, Colombo--Mingione
\cite{Colombo-Mingione-2015a,Colombo-Mingione-2015b},
Byun--Oh \cite{Byun-Oh-2020}, Byun--Ok--Song \cite{Byun-Ok-Song-2022},
De Filippis--Palatucci \cite{DeFilippis-Palatucci-2019}, and the references
therein. Nonuniformly elliptic and nonautonomous variational problems have also
been investigated by Beck--Mingione \cite{Beck-Mingione-2020},
De Filippis--Mingione \cite{DeFilippis-Mingione-2020,
DeFilippis-Mingione-2021}, and H\"ast\"o--Ok \cite{Hasto-Ok-2022}. Moreover,
double phase functionals belong to the broader class of variational integrals
with nonstandard growth, whose study goes back to the pioneering works of
Marcellini \cite{Marcellini-1989,Marcellini-1991}; see also the recent works
of Cupini--Marcellini--Mascolo \cite{Cupini-Marcellini-Mascolo-2023} and
Marcellini \cite{Marcellini-2023}.

A recent development in this direction concerns logarithmic double phase
operators. In \cite{Arora-CrespoBlanco-Winkert-2025b}, Arora--Crespo-Blanco--Winkert
introduced the logarithmic double phase functional
\[
    u \longmapsto
    \int_{\Omega}
    \left(
    \frac{|\nabla u|^{p(x)}}{p(x)}
    +
    \mu(x)\frac{|\nabla u|^{q(x)}}{q(x)}
    \log(e+|\nabla u|)
    \right)\,dx,
\]
and studied the corresponding logarithmic Musielak--Orlicz--Sobolev spaces.
They proved structural properties of these spaces, including separability,
reflexivity, embedding results, density of smooth functions, and the closedness
under truncations. They also established that the associated logarithmic double
phase operator is bounded, continuous, strictly monotone, of type $(S_+)$,
coercive and a homeomorphism, and obtained multiplicity results for equations
with superlinear right-hand sides. This logarithmic framework shows that double
phase theory naturally fits into the broader setting of Musielak--Orlicz
analysis and indicates the increasing relevance of nonstandard modular methods
in the study of nonlinear elliptic problems.

Another important feature of the problem considered in this paper is the
presence of a Kirchhoff-type nonlocal term. The origin of such problems goes
back to the classical model introduced by Kirchhoff \cite{Kirchhoff-1883},
\[
\rho \frac{\partial^2 \xi}{\partial t^2}
-
\left(
\frac{\rho_0}{h}
+
\frac{E}{2L}\int_0^L
\left|\frac{\partial \xi}{\partial z}\right|^2\,dz
\right)
\frac{\partial^2 \xi}{\partial z^2}=0,
\]
which describes the transverse vibrations of a stretched elastic string. This
equation is a nonlocal extension of the classical D'Alembert wave equation, since
the tension depends on the integral of the gradient of the displacement over the
whole interval. After the seminal contribution of Lions \cite{Lions-1978}, who
introduced an abstract framework for Kirchhoff-type equations, a large literature
has been devoted to the study of such problems by variational, topological and
operator-theoretic methods.

Kirchhoff-type equations have been investigated in many directions. Arosio and
Panizzi \cite{Arosio-Panizzi-1996} studied the well-posedness of the Kirchhoff
string, while Cavalcanti--Domingos Cavalcanti--Soriano
\cite{Cavalcanti-Cavalcanti-Soriano-2001} obtained global existence and decay
results for Kirchhoff--Carrier equations with nonlinear dissipation. More recent
developments include nonlinear Schr\"odinger--Kirchhoff equations with positive
solutions and ground states \cite{Chen-Fu-Wu-2020,Chen-Fu-Wu-2022}, as well as
non-autonomous Kirchhoff problems with Choquard-type nonlinearities
\cite{Zuo-Zhang-Radulescu-2024}. Kirchhoff-type problems with critical growth,
fractional and Orlicz--Sobolev structures, and generalized potential systems have
also been treated in \cite{Benchira-Matallah-ElMokhtar-Sabri-2023,
ChemsEddine-Ragusa-2022,ElHouari-Chadli-Moussa-2024}.

Recently, Kirchhoff problems have been combined with double-phase operators and
variable exponent growth. This direction is motivated by the fact that double
phase operators naturally describe media whose ellipticity changes from one
region to another, while Kirchhoff terms introduce a global nonlocal dependence
on the energy of the solution. Wang--Hou--Ge \cite{Wang-Hou-Ge-2021} studied
double phase problems by means of topological degree methods. Crespo-Blanco--
Gasi\'nski--Harjulehto--Winkert \cite{Crespo-Blanco-Gasinski-Harjulehto-Winkert-2022}
introduced and analyzed a new class of double-phase problems with variable
exponents and obtained existence and uniqueness results. Gasi\'nski--Winkert
\cite{Gasinski-Winkert-2020} considered double-phase problems with convection
terms and proved existence and uniqueness results through the theory of
pseudomonotone operators. More recently, Moujane--El Ouaarabi
\cite{Moujane-ElOuaarabi-2025} investigated Schr\"odinger--Kirchhoff double phase
problems with variable exponents and convection terms, proving the existence of
weak and strong generalized solutions by combining topological degree arguments
for generalized demicontinuous operators of type $(S_+)$ with the Galerkin method
in Musielak--Orlicz--Sobolev spaces. Very recently, Dixit--Mukherjee--Winkert \cite{Dixit-Mukherjee-Winkert-2026}
studied polyharmonic Kirchhoff problems with double phase structure and
subcritical nonlinearities, and developed a variational framework for obtaining
existence results in the corresponding Musielak--Orlicz--Sobolev setting. In the framework of higher-order Kirchhoff problems, Harrabi-Hamdani-Fiscella \cite{Harrabi-Hamdani-Fiscella-2024} established existence and multiplicity results for m-polyharmonic Kirchhoff equations without the Ambrosetti--Rabinowitz conditions. Their work provides an important contribution to the study of nonlocal higher-order problems and serves as one of the motivations for the present investigation.

Although these works provide substantial progress in the study of Kirchhoff
double phase equations, most of them are concerned with first-order operators.
The simultaneous presence of a Kirchhoff nonlocal coefficient, a double phase
growth structure and a polyharmonic operator creates additional difficulties.
Indeed, in the higher-order setting, standard truncation arguments and positivity
techniques are no longer directly available, while the Kirchhoff term couples the
whole equation through the modular energy of the highest-order gradient. This
makes the corresponding variational analysis substantially different from both
classical Kirchhoff problems and first-order double phase equations.

Let $\Omega\subset\mathbb R^N$ be a bounded domain with Lipschitz boundary
$\partial\Omega$, and let $m\in\mathbb N$. In this paper, we study the following
polyharmonic Kirchhoff double phase problem:
\begin{equation}\label{eq:main-problem}
\begin{cases}
M\!\left(\displaystyle\int_{\Om}\left(\frac{\abs{\nabla^{m}u}^{p}}{p}
+a(x)\frac{\abs{\nabla^{m}u}^{q}}{q}\right)\dd\right)\,
\mathcal{L}^{m}_{p,q,a}u
=h(x,u), & \text{in } \Om, \\[0.3cm]
u=\nabla u=\cdots=\nabla^{m-1}u=0, & \text{on } \partial\Om.
\end{cases}
\end{equation}
Throughout the paper, we assume that 
\begin{equation}\label{eq:standing}
1<p<q,\qquad mq<N,\qquad \frac{q}{p}\le \frac{N}{N-1},
\qquad
a\in C^{0,\kappa}(\overline{\Omega}),\quad a(x)\ge a_0>0 \qquad \text{for all }x\in\overline\Omega,
\end{equation}
for some $\kappa\in(0,1]$, and $h(x,u)$ is a function with different assumptions, specified in various sections further. 
We denote by
\[
p_m^*:=\frac{Np}{N-mp}\quad \text{and}\quad~q_m^*:=\frac{Nq}{N-mq}
\]
the critical Sobolev exponent corresponding to the $p$-phase, and $q$-phase, respectively.
The  $m$-order differential operator $\nabla^m u$ is defined by
\begin{equation}\label{eq:m-gradient}
\nabla^m u=
\begin{cases}
\nabla \Delta^{\frac{m-1}{2}}u, & \text{if } m \text{ is odd},\\[0.2cm]
\Delta^{\frac{m}{2}}u, & \text{if } m \text{ is even},
\end{cases}
\end{equation}
where $\nabla$ and $\Delta$ denote the classical gradient and Laplace operators, respectively. The polyharmonic double phase operator $\mathcal{L}^{m}_{p,q,a}(u)$ is given by
\begin{equation}\label{eq:operator-strong}
\mathcal{L}^{m}_{p,q,a}(u)=
\begin{cases}
-\operatorname{div}\!\Bigg(
\Delta^{\frac{m-1}{2}}
\Big(
|\nabla \Delta^{\frac{m-1}{2}}u|^{p-2}\nabla \Delta^{\frac{m-1}{2}}u
+
a(x)|\nabla \Delta^{\frac{m-1}{2}}u|^{q-2}\nabla \Delta^{\frac{m-1}{2}}u
\Big)
\Bigg),
& \text{for } m \text{ odd}, \\[0.4cm]
\Delta^{\frac{m}{2}}\!\Big(
|\Delta^{\frac{m}{2}}u|^{p-2}\Delta^{\frac{m}{2}}u
+
a(x)|\Delta^{\frac{m}{2}}u|^{q-2}\Delta^{\frac{m}{2}}u
\Big),
& \text{for } m \text{ even}.
\end{cases}
\end{equation}

We impose the following assumptions on the Kirchhoff function $M$:
\begin{enumerate}[label=\textnormal{(M\arabic*)},leftmargin=2.2cm]
\item\label{M1}
$M\colon [0,\infty)\to[0,\infty)$ is continuous, and there exist $t_0\ge 0$ and
\[
\gamma\in\left[1,\frac{p_m^*}{q}\right)
\]
such that
\[
tM(t)\le \gamma \mcM(t)\qquad \text{for all }t\ge t_0,
\]
where $\mcM(t)=\int_0^t M(\tau)\,\mathrm{d}\tau$.

\item\label{M2}
For every $\sigma>0$ there exists $k_\sigma>0$ such that
\[
M(t)\ge k_\sigma \qquad \text{for all }t\ge \sigma.
\]
\end{enumerate}
\begin{example}
	A concrete example of a Kirchhoff function satisfying hypothesis \ref{M1}, and \ref{M2} is
	\begin{align}\label{4.1.2}
		{M}(\tau)=a+b{\gamma}\tau^{\gamma-1}, \quad a,b\ge0,\quad a+b>0,\quad
		\gamma \in \begin{cases}
			\left(1,\frac{p_m^*}{q}\right), & \text{if } b>0,\\[4pt]
			1, & \text{if } b=0.
		\end{cases}
	\end{align}
	If ${M}$ is given by \eqref{4.1.2}, the corresponding problem \eqref{eq:main-problem} is called nondegenerate when $a>0$ and $b\ge0$, while it is called degenerate when $a=0$ and $b>0$. 
 \end{example}
We now summarize the main results of the paper. Under different assumptions on the
nonlinearity, we obtain the following results
\begin{itemize}
    \item Problem \eqref{eq:main-problem} possesses infinitely many distinct pairs
    of weak solutions $\{u_j,-u_j\}_{j\in\N}$, when $h(x,u)= K(x)g(u)$.

    \item Under various growth assumptions on $h(x,u)$, Problem \eqref{eq:main-problem-sec4} can possess a nontrivial weak solution
    of mountain pass type or infinitely many distinct pairs of
    nontrivial weak solutions.
\end{itemize}

The paper is organized as follows. In Section \ref{sec:preliminaries}, we collect some basic properties of Musielak--Orlicz spaces and establish the functional framework for the polyharmonic Kirchhoff double phase problem. In Section 3, we prove the compactness properties of the associated energy functional and verify the symmetric mountain-pass geometry in order to obtain infinitely many weak solutions without the Ambrosetti--Rabinowitz condition. In Section 4, we derive further existence results for the general nonlinearity by using mountain-pass arguments and minimization techniques. In particular, we also discuss the $\gamma p$-sublinear case and obtain both multiplicity and nontrivial existence results under suitable behavior of the primitive near the origin and at infinity.
\section{Preliminaries and functional framework}\label{sec:preliminaries}

In this section, we recall the basic notions and properties of generalized Orlicz and Musielak--Orlicz Sobolev spaces that will be used throughout the paper.
For $1\le s<\infty$, we denote by $L^s(\Om)$ the Lebesgue space endowed with the norm $\norm{\cdot}_s$. For $1<s<\infty$, the Sobolev space $W_0^{m,s}(\Om)$ is equipped with the equivalent norm $\norm{\nabla^m(\cdot)}_s$.

\begin{definition}\label{def:phi}
A function $\varphi\colon [0,\infty)\to[0,\infty)$ is called a $\Phi$-function if it is continuous, convex, satisfies $\varphi(0)=0$, and $\varphi(t)>0$ for all $t>0$.
\end{definition}

\begin{definition}\label{def:N}
A $\Phi$-function $\varphi$ is said to be an $\mathcal{N}$-function if
\[
\lim_{t\to 0^+}\frac{\varphi(t)}{t}=0
\quad \text{and} \quad
\lim_{t\to\infty}\frac{\varphi(t)}{t}=\infty.
\]
\end{definition}

\begin{definition}\label{def:gen-phi}
A function $\varphi\colon \Om\times[0,\infty)\to[0,\infty)$ is called a generalized $\Phi$-function, denoted by $\varphi\in \Phi(\Om)$, if the map $x\mapsto \varphi(x,t)$ is measurable for every $t\ge 0$, and for a.e.\ $x\in\Om$, the function $t\mapsto \varphi(x,t)$ is a $\Phi$-function.
\end{definition}

\begin{definition}\label{def:gen-N}
A function $\varphi\colon \Om\times[0,\infty)\to[0,\infty)$ is called a generalized $\mathcal{N}$-function, written $\varphi\in N(\Om)$, if for each $t\ge 0$ the map $x\mapsto \varphi(x,t)$ is measurable and, for a.e.\ $x\in\Om$, the function $t\mapsto \varphi(x,t)$ is an $\mathcal{N}$-function.
\end{definition}

\begin{definition}\label{def:dominance}
Let $\varphi,\psi\in \Phi(\Om)$. We say that $\varphi$ is dominated by $\psi$, and write $\varphi\preceq \psi$, if there exist constants $C_1,C_2>0$ and a function $h\in L^1(\Om)$ such that
\[
\varphi(x,t)\le C_1\,\psi(x,C_2 t)+h(x)
\]
for a.e.\ $x\in\Om$ and all $t\ge 0$.
\end{definition}

The following embedding result is classical; see \cite[Theorem 8.5]{Musielak-1983}.
\begin{proposition}\label{prop:embedding-general}
Let $\varphi,\psi\in N(\Om)$ be such that $\varphi\preceq \psi$. Then the embedding
\[
L^\psi(\Om)\hookrightarrow L^\varphi(\Om)
\]
is continuous.
\end{proposition}
We now introduce the double-phase density
\[
\mcH(x,t)=t^p+a(x)t^q,
\qquad \text{for all}~(x,t)\in \Om\times[0,\infty),
\]
which defines the growth of the operator under consideration. Under assumption \eqref{eq:standing}, the function $\mcH$ is a generalized $\mathcal{N}$-function.
We define the modular
\[
\varrh(u):=\int_{\Om}\mcH(x,|u|)\,\dd
=
\int_{\Om}\bigl(|u|^p+a(x)|u|^q\bigr)\,\dd,
\]
and the Musielak--Orlicz space
\[
L^{\mcH}(\Om)=
\left\{
u:\Om\to\R \text{ measurable }:\ \varrh(u)<\infty
\right\}.
\]
This space is endowed with the Luxemburg norm
\[
\norm{u}_{\mcH}
=
\inf\left\{
\lambda>0:\ \varrh\!\left(\frac{u}{\lambda}\right)\le 1
\right\}.
\]
We also introduce the weighted Lebesgue space
\[
L_a^q(\Om)=
\left\{
u:\Om\to\R \text{ measurable }:\ \int_{\Om}a(x)|u|^q\,\dd<\infty
\right\},
\]
with seminorm
\[
\norm{u}_{q,a}=
\left(\int_{\Om}a(x)|u|^q\,\dd\right)^{1/q}.
\]

The relation between the modular $\varrh$ and the Luxemburg norm is given below.
\begin{proposition}\label{prop:modular-norm}
Let $u\in L^{\mcH}(\Om)$. Then:
\begin{enumerate}[label=\textnormal{(\roman*)}]
\item $\norm{u}_{\mcH}=c$ if and only if $\varrh(u/c)=1$;
\item $\norm{u}_{\mcH}<1$ (resp.\ $=1$, $>1$) if and only if $\varrh(u)<1$ (resp.\ $=1$, $>1$);
\item if $\norm{u}_{\mcH}<1$, then $\norm{u}_{\mcH}^q \le \varrh(u)\le \norm{u}_{\mcH}^p$;
\item if $\norm{u}_{\mcH}>1$, then $\norm{u}_{\mcH}^p \le \varrh(u)\le \norm{u}_{\mcH}^q$;
\item $\norm{u}_{\mcH}\to 0$ if and only if $\varrh(u)\to 0$;
\item $\norm{u}_{\mcH}\to\infty$ if and only if $\varrh(u)\to\infty$.
\end{enumerate}
\end{proposition}
In particular,
\begin{equation}\label{eq:min-max-LH}
\min\{\norm{u}_{\mcH}^p,\norm{u}_{\mcH}^q\}
\le
\int_{\Om}(|u|^p+a(x)|u|^q)\,\dd
\le
\max\{\norm{u}_{\mcH}^p,\norm{u}_{\mcH}^q\}.
\end{equation}
\begin{definition}\label{definitionsobolev}
    The Musielak-Orlicz Sobolev space $W^{m,\mathcal{H}}(\Omega)$ consists of all functions $u \in L^\mathcal{H}(\Omega)$ such that $|\nabla^k u| \in L^\mathcal{H}(\Omega)$ for all $k\in\{0,1,2,\cdots,m\}$, endowed with the norm 
\[
\|u\|_{m,\mathcal{H}} = \sum_{k=0}^{m}\|\nabla^ku\|_\mathcal{H}, \]
where we have used the notation $\|\nabla^k u\|_\mathcal{H} = \||\nabla^k u|\|_\mathcal{H}$ for all $k\in\{0,1,2,\cdots,m\}$. It forms a reflexive Banach space. Moreover, the space $W_0^{m,\mathcal{H}}(\Omega)$ is defined as the completion of $C_0^\infty(\Omega)$ with respect to the norm of $W^{m,\mathcal{H}}(\Omega)$, which is also a reflexive Banach space.
\end{definition}

The following Poincar\'{e}-type inequality can be found in Colasuonno--Squassina \cite{Colasuonno-Squassina-2016} and Crespo-Blanco--Gasi\'{n}ski--Harjulehto--Winkert \cite{Crespo-Blanco-Gasinski-Harjulehto-Winkert-2022}.
\begin{theorem}\cite[Theorem 2.2.]{Dixit-Mukherjee-Winkert-2026}\label{thm:poincare}
There exists $C>0$ such that
\[
\norm{u}_{\mcH}\le C \norm{\nabla u}_{\mcH}
\qquad \text{for all } u\in W_0^{1,\mcH}(\Om).
\]
\end{theorem}

\begin{corollary}\cite[Corollary 2.3]{Dixit-Mukherjee-Winkert-2026}\label{cor:equiv-norm}
The norm $\norm{u}=\norm{\nabla^m u}_{\mcH}$ is equivalent to $\norm{u}_{m,\mcH}$ on $W_0^{m,\mcH}(\Om)$.
\end{corollary}

Hence, we shall use $\norm{u}=\norm{\nabla^m u}_{\mcH}$ in $W_0^{m,\mcH}(\Om)$. Moreover,
\begin{equation}\label{eq:min-max-WH}
\min\{\norm{u}^p,\norm{u}^q\}
\le
\int_{\Om}(|\nabla^m u|^p+a(x)|\nabla^m u|^q)\,\dd
\le
\max\{\norm{u}^p,\norm{u}^q\}.
\end{equation}

\begin{proposition}\cite[Proposition 2.10.]{Dixit-Mukherjee-Winkert-2026}\label{prop:embeddings}
Under \eqref{eq:standing}, the following hold:
\begin{enumerate}[label=\textnormal{(\roman*)}]
\item $L^{\mcH}(\Om)\hookrightarrow L^s(\Om)$ and $W_0^{m,\mcH}(\Om)\hookrightarrow W_0^{m,s}(\Om)$ continuously for $s\in[1,p]$;
\item $W_0^{m,\mcH}(\Om)\hookrightarrow L^s(\Om)$ continuously for $s\in[1,p_m^*]$ and compactly for $s\in[1,p_m^*)$;
\item $L^q(\Om)\hookrightarrow L^{\mcH}(\Om)\hookrightarrow L_a^q(\Om)$ continuously.
\end{enumerate}
\end{proposition}

\begin{remark}
By Proposition \ref{prop:embeddings} (ii), for each $s\in[1,p_m^*]$, there exists $\mcS_s>0$ such that
\begin{equation}\label{embedding}
\norm{u}_s \le \mcS_s \norm{u}
\qquad \text{for all } u\in W_0^{m,\mcH}(\Om).
\end{equation}
\end{remark}

We conclude this section by collecting some auxiliary estimates which will be repeatedly used in the sequel.
The following estimates follow from assumptions \ref{M1} and \ref{g2}. According to \ref{M1}, there exists a constant $C_0>0$ such that
\begin{equation}\label{eq:M-estimate}
\gamma \mcM(t)-tM(t)\ge -C_0
\qquad \text{for all } t\ge 0.
\end{equation}
Moreover, from \ref{g2} for every $\varepsilon>0$ there exists $C_\varepsilon>0$ such that
\begin{equation}\label{eq:g-estimate}
|g(t)|\le \varepsilon |t|^{p_m^*-1}+C_\varepsilon,
\qquad
|G(t)|\le \varepsilon |t|^{p_m^*}+C_\varepsilon
\qquad \text{for all } t\in\R~.
\end{equation}
Furthermore, for all $t,w\in\R$, we derive 
\begin{equation}\label{eq:g-difference}
|g(t)(t-w)|
\le
\varepsilon\bigl(|t|^{p_m^*}+|t-w|^{p_m^*}\bigr)+C_\varepsilon|t-w|.
\end{equation}

\begin{definition}
\begin{enumerate}
    \item The functional $J$ is said to satisfy the Palais--Smale condition, abbreviated $(PS)$, if any sequence $\{u_n\}\subset W_0^{m,\mcH}(\Om)$ such that $\{J(u_n)\}$ is bounded and
\[
J'(u_n)\to 0 \quad \text{in } \bigl(W_0^{m,\mcH}(\Om)\bigr)^*
\]
admits a strongly convergent subsequence.
\item The functional $J$ is said to satisfy the Cerami condition, abbreviated $(C)$, if any sequence $\{u_n\}\subset W_0^{m,\mcH}(\Om)$ such that $\{J(u_n)\}$ is bounded and
\[
(1+\|u_n\|)\|J'(u_n)\|\to 0
\]
admits a strongly convergent subsequence.
\end{enumerate}
\end{definition}
Multiple solutions for \eqref{eq:main-problem} are provided by abstract results, like the following symmetric mountain pass theorem. 
 \begin{theorem}\cite{Rabinowitz-1986}\label{symmetricmountainpasstheorem}
Let $E$ be a real infinite-dimensional Banach space and let $I\in C^1(E,\R)$ be an even functional satisfying either the $(PS)$ condition or the $(C)$ condition. Assume that $E=E^-\oplus E^+$, where $E^-$ is finite dimensional, and that:
\begin{enumerate}[label=\textnormal{(\roman*)}]
\item $I(0)=0$;
\item there exist $\alpha,\rho>0$ such that
\[
I(u)\ge \alpha
\quad \text{for all } u\in E^+ \text{ with } \|u\|=\rho;
\]
\item for every finite-dimensional subspace $W\subset E$, there exists $R=R(W)>0$ such that
\[
I(u)\le 0
\quad \text{for all } u\in W \text{ with } \|u\|\ge R.
\]
\end{enumerate}
Then $I$ admits an unbounded sequence of critical values.
\end{theorem}
We also recall some basic notation of Krasnoselskii's genus theory, which will be used later in the multiplicity results.
Let E be a Banach space and
\[
\Gamma:=\{A\subset E\setminus\{0\}:\ A \text{ is closed and symmetric with respect to 0}\}.
\]

\begin{definition}\cite{Rabinowitz-1986}
For $A\in \Gamma$, the genus $\gamma(A)$ is defined as the least integer $k\in\N$ such that there exists an odd continuous mapping
\[
\phi:A\to \R^k\setminus\{0\}.
\]
\end{definition}

\begin{theorem}\cite{Rabinowitz-1986}\label{genustheorem}
Let $I\in C^1(E,\R)$ be an even functional and satisfy the $(PS)$ condition. For each $k\in\N^*$, and $c\in\mathbb{R}$ set
\[
\Gamma_k:=\{A\in\Gamma:\ \gamma(A)\ge k\},
\qquad
c_k:=\inf_{A\in\Gamma_k}\sup_{u\in A}I(u).
\]
Then:
\begin{enumerate}[label=\textnormal{(\roman*)}]
\item if $\Gamma_k\neq\emptyset$ and $-\infty<c_k<0$, then $c_k$ is a critical value of $I$;
\item if $c_k=\cdots=c_{k+\ell}=c$ for some $\ell\in\N$ and $c\neq I(0)$, then
\[
\gamma\bigl(\{u\in E:\ I'(u)=0,\ I(u)=c\}\bigr)\ge \ell+1.
\]
\end{enumerate}
\end{theorem}

\section{Existence of Infinitely Many Solutions}
In this section, we establish the compactness and geometric properties of the energy functional associated with problem \eqref{eq:main-problem} with $h(x,u)=K(x)g(u)$, where 
\begin{equation}\label{eq:weightK}
K\in L^\infty(\Omega),
\qquad
K(x)>0 \quad \text{for a.e. }x\in\Omega.
\end{equation}
Concerning the nonlinearity $g$, we assume that it is continuous and satisfies:
\begin{enumerate}[label=\textnormal{(g\arabic*)},leftmargin=2.2cm]
\item\label{g1}
There exists $C>0$ such that
\[
C|g(t)|^{(p_m^*)'}\le tg(t)-\gamma q\,G(t)
\qquad \text{for all }t\in\R,
\]
where
\[
G(t)=\int_0^t g(s)\,\mathrm{d}s,
\qquad
(p_m^*)'=\frac{p_m^*}{p_m^*-1}.
\]

\item\label{g2}
\[
\lim_{|t|\to\infty}\frac{g(t)}{|t|^{p_m^*-1}}=0~~.
\]

\item\label{g3}
$g$ is odd, i.e.,
\[
g(-t)=-g(t)\qquad \text{for all }t\in\R~
\]

\item\label{g4}
\[
\lim_{|t|\to\infty}\frac{G(t)}{|t|^{\gamma q}}=\infty~~.
\]
\end{enumerate}
First, we prove that the functional satisfies the appropriate compactness condition. Then we verify the symmetric mountain-pass geometry and obtain infinitely many weak solutions.

\begin{definition}\label{de-new}
A function $u\in W_0^{m,\mcH}(\Om)$ is said to be a weak solution of \eqref{eq:main-problem} if
\begin{equation}\label{eq:weak-solution}
M\!\left(\Phi_{\mcH}(\nabla^m u)\right)
\int_{\Om}
\Big(
|\nabla^m u|^{p-2}\nabla^m u
+a(x)|\nabla^m u|^{q-2}\nabla^m u
\Big)\cdot \nabla^m\varphi\,\dd
=
\int_{\Om}K(x)g(u)\varphi\,\dd
\end{equation}
for all $\varphi\in W_0^{m,\mcH}(\Om)$, where
\begin{equation}\label{eq:PhiH}
\Phi_{\mcH}(\nabla^m u)
=
\int_{\Om}\left(\frac{|\nabla^m u|^p}{p}
+a(x)\frac{|\nabla^m u|^q}{q}\right)\dd.
\end{equation}
\end{definition}
The associated energy functional $J\colon W_0^{m,\mcH}(\Om)\to\R$ is given by
\begin{equation}\label{eq:energy-functional}
J(u)
=
\mcM\!\left(\Phi_{\mcH}(\nabla^m u)\right)
-
\int_{\Om}K(x)G(u)\,\dd.
\end{equation}
It is well known that critical points of $J$ correspond to weak solutions of \eqref{eq:main-problem}.

\begin{lemma}\cite[Proposition 2.3]{ElAhmadi-Ayoujil-Berrajaa-2024}\label{lem:S+}
Let $\Psi(u)=\Phi_{\mcH}(\nabla^m u)$, recall \eqref{eq:PhiH}. Then
$\Psi'\colon W_0^{m,\mcH}(\Om)\to (W_0^{m,\mcH}(\Om))^*$
is bounded, strictly monotone, and of type $(S_+)$.
\end{lemma}


\begin{proposition}\label{prop:PS-C}
Assume \eqref{eq:standing}, \eqref{eq:weightK}, \textnormal{\ref{M1}--\ref{M2}} and
\textnormal{\ref{g1}--\ref{g2}}. Then:
\begin{enumerate}[label=\textnormal{(\roman*)}]
\item if \(p\ge 2\), the functional \(J\) satisfies the Palais--Smale condition in \(W_0^{m,\mcH}(\Om)\);
\item if \(1<p<2\), the functional \(J\) satisfies the Cerami condition in \(W_0^{m,\mcH}(\Om)\).
\end{enumerate}
\end{proposition}

\begin{proof}
We prove \textnormal{(i)} by considering two cases. Let \(\{u_n\}\subset W_0^{m,\mcH}(\Om)\) be a Palais--Smale sequence for \(J\), that is,
\begin{equation}\label{eq:PS-seq-h}
\sup_{n\in\N}|J(u_n)|<\infty,
\qquad
J'(u_n)\to0
\quad \text{in } \bigl(W_0^{m,\mcH}(\Om)\bigr)^*.
\end{equation}

\noindent\textit{Case 1.} \(\inf_{n\in\N}\|u_n\|=d>0\).

Define
\[
\Phi_n:=\Phi_{\mcH}(\nabla^m u_n)
=
\int_{\Om}\left(\frac{|\nabla^m u_n|^p}{p}
+a(x)\frac{|\nabla^m u_n|^q}{q}\right)\dd,
\]
and
\[
A_n:=\int_{\Om}\bigl(|\nabla^m u_n|^p+a(x)|\nabla^m u_n|^q\bigr)\dd.
\]
By \eqref{eq:min-max-WH},
\[
\Phi_n\ge \frac1q\min\{\|u_n\|^p,\|u_n\|^q\}
\ge \frac1q\min\{d^p,d^q\}=:\sigma_d.
\]
Hence, by \textnormal{\ref{M2}}, there exists \(k_d>0\) such that
\begin{equation}\label{eq:M-lower-h}
M(\Phi_n)\ge k_d
\qquad \text{for all } n\in\N.
\end{equation}

We show that \(\{u_n\}\) is bounded in \(W_0^{m,\mcH}(\Om)\). Since
\[
\ip{J'(u_n)}{u_n}
=
M(\Phi_n)A_n-\int_{\Om}K(x)g(u_n)u_n\,\dd,
\]
it follows from \eqref{eq:M-lower-h} and \eqref{eq:min-max-WH} that
\begin{equation}\label{eq:basic-bound-h}
k_d\min\{\|u_n\|^p,\|u_n\|^q\}
\le
\bigl|\ip{J'(u_n)}{u_n}\bigr|
+
\int_{\Om}K(x)|g(u_n)||u_n|\,\dd.
\end{equation}

Moreover,
\begin{align}
\int_{\Om}K(x)\bigl(g(u_n)u_n-\gamma q\,G(u_n)\bigr)\,\dd
&=
\gamma q\,J(u_n)-\ip{J'(u_n)}{u_n}
+\bigl(M(\Phi_n)A_n-\gamma q\,\mcM(\Phi_n)\bigr).
\label{eq:key-id-h}
\end{align}
Since \(A_n\le q\Phi_n\), \eqref{eq:M-estimate} yields
\[
M(\Phi_n)A_n-\gamma q\,\mcM(\Phi_n)
\le
q\Phi_n M(\Phi_n)-\gamma q\,\mcM(\Phi_n)
\le qC_0.
\]
Hence, using \eqref{eq:PS-seq-h} and
\[
\bigl|\ip{J'(u_n)}{u_n}\bigr|
\le \|J'(u_n)\|\,\|u_n\|
=o(1)\|u_n\|,
\]
we obtain
\begin{equation}\label{eq:gq-bound-h}
\int_{\Om}K(x)\bigl(g(u_n)u_n-\gamma q\,G(u_n)\bigr)\,\dd
\le
C_1+o(1)\|u_n\|
\end{equation}
for some constant \(C_1>0\).

By \textnormal{\ref{g1}}, there exists \(C_2>0\) such that
\[
C_2|g(t)|^{(p_m^*)'}
\le t\,g(t)-\gamma q\,G(t)
\qquad \text{for all } t\in\R.
\]
Multiplying by \(K(x)\ge0\) and integrating, we infer from \eqref{eq:gq-bound-h} that
\begin{equation}\label{eq:gLp-bound-h}
\int_{\Om}K(x)|g(u_n)|^{(p_m^*)'}\,\dd
\le
C_3\bigl(1+\|u_n\|\bigr)
\end{equation}
for some constant \(C_3>0\).

Using H\"older's inequality, \(K\in L^\infty(\Om)\), and Proposition \ref{prop:embeddings}, we have
\[
\int_{\Om}K(x)|g(u_n)||u_n|\,\dd
\le
C\left(\int_{\Om}K|g(u_n)|^{(p_m^*)'}\,\dd\right)^{\frac1{(p_m^*)'}}
\|u_n\|.
\]
Combining this with \eqref{eq:basic-bound-h}, \eqref{eq:gLp-bound-h}, and \eqref{eq:PS-seq-h}, we obtain
\begin{equation}\label{eq:main-boundedness-h}
k_d\min\{\|u_n\|^p,\|u_n\|^q\}
\le
o(1)\|u_n\|
+
C\bigl(1+\|u_n\|\bigr)^{\frac1{(p_m^*)'}}\|u_n\|.
\end{equation}

Suppose, by contradiction, that \(\|u_n\|\to\infty\). Then, for \(n\) large enough,
\[
\min\{\|u_n\|^p,\|u_n\|^q\}=\|u_n\|^p.
\]
Dividing \eqref{eq:main-boundedness-h} by \(\|u_n\|^p\), we get
\[
k_d
\le
o(1)\|u_n\|^{1-p}
+
C\bigl(1+\|u_n\|\bigr)^{\frac1{(p_m^*)'}}\|u_n\|^{1-p}.
\]
Since
\[
1-p+\frac1{(p_m^*)'}<0,
\]
the right-hand side tends to \(0\), which contradicts \(k_d>0\). Hence \(\{u_n\}\) is bounded in \(W_0^{m,\mcH}(\Om)\).

Therefore, passing to a subsequence, there exists \(u\in W_0^{m,\mcH}(\Om)\) such that
\[
u_n\rightharpoonup u
\quad \text{in } W_0^{m,\mcH}(\Om).
\]
By Proposition \ref{prop:embeddings}, we also have
\[
u_n\to u
\quad \text{in } L^1(\Om)
\quad \text{and a.e. in } \Om.
\]

We now prove the strong convergence. Observe that
\[
\ip{J'(u_n)}{u_n-u}
=
M(\Phi_n)\ip{\Psi'(u_n)}{u_n-u}
-
\int_{\Om}K(x)g(u_n)(u_n-u)\,\dd.
\]
Since \(J'(u_n)\to0\), it is enough to show that
\begin{equation}\label{eq:reaction-vanish-h}
\int_{\Om}K(x)g(u_n)(u_n-u)\,\dd\to0.
\end{equation}

Hence, by H\"older's inequality, \eqref{eq:g-estimate} and Proposition \ref{prop:embeddings},
\begin{align*}
\int_{\Om}K(x)|g(u_n)||u_n-u|\,\dd
&\le
\varepsilon \|K\|_\infty
\int_{\Om}|u_n|^{p_m^*-1}|u_n-u|\,\dd
+
C_\varepsilon \|K\|_\infty \|u_n-u\|_{1} \\
&\le
\varepsilon C
+
C_\varepsilon \|K\|_\infty \|u_n-u\|_{1},
\end{align*}
where \(C>0\) is independent of \(n\) and \(\varepsilon\). Since \(u_n\to u\) in \(L^1(\Om)\), we obtain
\[
\limsup_{n\to\infty}
\int_{\Om}K(x)|g(u_n)||u_n-u|\,\dd
\le \varepsilon C.
\]
Letting \(\varepsilon\to0\), we get \eqref{eq:reaction-vanish-h}.

Consequently,
\[
M(\Phi_n)\ip{\Psi'(u_n)}{u_n-u}\to0.
\]
If \(\Phi_n\to0\), then \(\|u_n\|\to0\) by \eqref{eq:min-max-WH}, and hence \(u_n\to0\) strongly in \(W_0^{m,\mcH}(\Om)\).

Otherwise, there exist \(\delta>0\) and a subsequence, still denoted by \(\{u_n\}\), such that
\[
\Phi_n\ge \delta>0.
\]
By \textnormal{\ref{M2}}, there exists \(k_\delta>0\) such that
\[
M(\Phi_n)\ge k_\delta>0.
\]
Therefore,
\[
\ip{\Psi'(u_n)}{u_n-u}\to0.
\]
Since \(\Psi'\) is of type \((S_+)\) by Lemma \ref{lem:S+}, we conclude that
\[
u_n\to u
\quad \text{strongly in } W_0^{m,\mcH}(\Om).
\]
Thus \(J\) satisfies the Palais--Smale condition.

\noindent\textit{Case 2.} \(\inf_{n\in\N}\|u_n\|=0\).

If \(0\) is an accumulation point of \(\{\|u_n\|\}\), then there exists a subsequence, still denoted by \(\{u_n\}\), such that
\[
\|u_n\|\to0.
\]
Hence \(u_n\to0\) strongly in \(W_0^{m,\mcH}(\Om)\). Otherwise, there exists \(d>0\) such that
\[
\|u_n\|\ge d
\]
for all sufficiently large \(n\), and therefore Case 1 applies. This proves \textnormal{(i)}.

We prove \textnormal{(ii)} similarly. Let \(\{u_n\}\subset W_0^{m,\mcH}(\Om)\) be a Cerami sequence for \(J\), that is,
\[
\sup_{n\in\N}|J(u_n)|<\infty,
\qquad
(1+\|u_n\|)\|J'(u_n)\|\to0.
\]
Then
\[
\bigl|\ip{J'(u_n)}{u_n}\bigr|
\le \|J'(u_n)\|\,\|u_n\|
\le (1+\|u_n\|)\|J'(u_n)\|\to0.
\]
Hence, in \eqref{eq:gq-bound-h}, the term \(o(1)\|u_n\|\) is replaced by \(o(1)\). Consequently, instead of \eqref{eq:gLp-bound-h}, we obtain
\[
\int_{\Om}K(x)|g(u_n)|^{(p_m^*)'}\,\dd
\le C
\]
for some constant \(C>0\). Therefore, arguing as in Case 1, we get
\[
k_d\min\{\|u_n\|^p,\|u_n\|^q\}
\le
o(1)+C\|u_n\|.
\]
If \(\|u_n\|\to\infty\), then, for \(n\) large enough,
\[
\min\{\|u_n\|^p,\|u_n\|^q\}=\|u_n\|^p.
\]
Dividing the above inequality by \(\|u_n\|^p\), we obtain
\[
k_d
\le
o(1)\|u_n\|^{-p}+C\|u_n\|^{1-p}.
\]
Since \(p>1\), the right-hand side tends to \(0\), which contradicts \(k_d>0\). Hence \(\{u_n\}\) is bounded in \(W_0^{m,\mcH}(\Om)\).

Once boundedness is established, we have
\[
\|J'(u_n)\|\to0.
\]
Thus, the strong convergence argument used in \textnormal{(i)} applies without any change. Therefore \(J\) satisfies the Cerami condition in \(W_0^{m,\mcH}(\Om)\).
\end{proof}

\begin{lemma}\label{lem:geometry}
Assume \textnormal{\ref{g2}} and \textnormal{\ref{M2}} hold. Then, for every $\rho>0$, there exists a finite-dimensional subspace $E^- \subset W_0^{m,\mcH}(\Om)$ and a constant $\alpha>0$ such that
\[
J(u)\ge \alpha
\qquad \text{for all } u\in E^+ \text{ with } \|u\|=\rho,
\]
where $E^+$ is a topological complement of $E^-$, that is,
\[
W_0^{m,\mcH}(\Om)=E^- \oplus E^+.
\]
\end{lemma}

\begin{proof}
Let $\{e_j\}_{j\in\N}$ be a Schauder basis of $W_0^{m,\mcH}(\Om)$ (see \cite[Lemma 6.5]{Arora-CrespoBlanco-Winkert-2025a}). Then each
$u\in W_0^{m,\mcH}(\Om)$ admits a unique representation
\[
u=\sum_{j=1}^\infty a_j e_j.
\]
For each $k\in\N^*$, set
\[
E_k:=\mathrm{span}\{e_1,\dots,e_k\},
\]
and let $P_k:W_0^{m,\mcH}(\Om)\to E_k$ be the corresponding continuous linear projection,
\[
P_k(u)=\sum_{j=1}^k a_j e_j.
\]
Denote by
\[
F_k:=\ker(P_k).
\]
Then $F_k$ is a topological complement of $E_k$, namely
\[
W_0^{m,\mcH}(\Om)=E_k\oplus F_k.
\]

Fix $\rho>0$ and define
\[
S_k(\rho):=\{u\in F_k:\|u\|=\rho\},
\qquad
\beta_k:=\sup_{u\in S_k(\rho)}\int_{\Om}K(x)|G(u)|\,dx.
\]
We claim that
\begin{equation}\label{eq:beta-to-0}
\beta_k\to0
\qquad \text{as } k\to\infty.
\end{equation}

Assume by contradiction that there exist $\delta_0>0$ and $u_k\in S_k(\rho)$ such that
\[
\int_{\Om}K(x)|G(u_k)|\,dx \ge \delta_0
\qquad \text{for all }k\in\mathbb N.
\]
Since $\|u_k\|=\rho$, the sequence $\{u_k\}$ is bounded in $W_0^{m,\mcH}(\Om)$. Hence, up to a subsequence,
\[
u_k \rightharpoonup u
\quad \text{in } W_0^{m,\mcH}(\Om),
\]
\[
u_k\to u
\quad \text{in } L^s(\Om)\ \text{for all } s<p_m^*,
\]
and
\[
u_k(x)\to u(x)
\quad \text{a.e. in } \Om.
\]

Fix $j\in\N^*$. Since
\[
P_j\circ P_k=P_j
\qquad \text{for all } k\ge j,
\]
we have $F_k\subset F_j$ whenever $k\ge j$. Therefore,
\[
P_j(u_k)=0
\qquad \text{for all } k\ge j.
\]
Passing to the limit and using the continuity of $P_j$, we obtain
\[
P_j(u)=0
\qquad \text{for every } j\in\N^*.
\]
Hence,
\[
u=\lim_{j\to\infty}P_j(u)=0.
\]
Consequently,
\[
u_k\to0
\qquad \text{a.e. in } \Om.
\]

Now, by the estimate \eqref{eq:g-estimate}
\begin{equation}\label{eq:G-eps-used}
|G(t)|\le \varepsilon |t|^{p_m^*}+C_\varepsilon
\qquad \text{for all } t\in\R,
\end{equation}
so for every measurable set $A\subset\Om$ we have
\[
\int_A K(x)|G(u_k)|\,dx
\le
\|K\|_\infty
\left(
\varepsilon \int_{\Om}|u_k|^{p_m^*}\,dx + C_\varepsilon |A|
\right).
\]
Since $\{u_k\}$ is bounded in $L^{p_m^*}(\Om)$, the family $\{K(x)|G(u_k)|\}$ is uniformly integrable. Together with the pointwise convergence and continuity of $G(t)$
\[
K(x)|G(u_k)|\to0
\qquad \text{a.e. in } \Om,
\]
Thus, Vitali's theorem yields
\[
\int_{\Om}K(x)|G(u_k)|\,dx \to0,
\]
which contradicts the choice of $\delta_0$. Therefore \eqref{eq:beta-to-0} holds.

We now estimate the Kirchhoff term from below. Set
\[
\eta_\rho:=\frac1q\min\{\rho^p,\rho^q\}>0.
\]
If $\|u\|=\rho$, then by \eqref{eq:min-max-WH},
\[
\Phi_{\mcH}(\nabla^m u)
=
\int_{\Om}\left(\frac{|\nabla^m u|^p}{p}+a(x)\frac{|\nabla^m u|^q}{q}\right)\,dx
\ge
\frac1q\int_{\Om}\bigl(|\nabla^m u|^p+a(x)|\nabla^m u|^q\bigr)\,dx
\ge \eta_\rho.
\]
Let
\[
\delta_\rho:=\frac{\eta_\rho}{2}.
\]
By \textnormal{\ref{M2}}, there exists $m_{\delta_\rho}>0$ such that
\[
M(t)\ge m_{\delta_\rho}
\qquad \text{for all } t\ge \delta_\rho.
\]
Hence,
\[
\mcM(\Phi_{\mcH}(\nabla^m u))
\ge
\mcM(\eta_\rho)
=
\int_0^{\eta_\rho}M(t)\,dt
\ge
\int_{\delta_\rho}^{\eta_\rho}M(t)\,dt
\ge
m_{\delta_\rho}\frac{\eta_\rho}{2}
=:c_\rho>0.
\]

Therefore, for every $u\in S_k(\rho)$,
\[
J(u)
=
\mcM(\Phi_{\mcH}(\nabla^m u))
-
\int_{\Om}K(x)G(u)\,dx
\ge
c_\rho-\int_{\Om}K(x)|G(u)|\,dx
\ge
c_\rho-\beta_k.
\]
By \eqref{eq:beta-to-0}, we may choose $k_0$ sufficiently large such that
\[
\beta_{k_0}\le \frac{c_\rho}{2}.
\]
Setting
\[
E^-:=E_{k_0},
\qquad
E^+:=F_{k_0},
\qquad
\alpha:=\frac{c_\rho}{2},
\]
we obtain
\[
J(u)\ge \alpha
\qquad \text{for all } u\in E^+ \text{ with } \|u\|=\rho.
\]
This completes the proof.
\end{proof}
We now present our first main result.
\begin{theorem}\label{thm:multiple-solutions}
Assume that \textnormal{\ref{M1}--\ref{M2}} hold and that
$g$ satisfies \textnormal{\ref{g1}--\ref{g4}}. Then the functional
$J$ admits an unbounded sequence of critical values. In particular, in the separable case $h(x,u)=K(x)g(u)$,
problem \eqref{eq:main-problem} possesses infinitely many distinct pairs
of weak solutions
\[
\{u_j,-u_j\}_{j\in\N}.
\]
\end{theorem}

\begin{proof}
We verify the assumptions of the symmetric mountain pass theorem \ref{symmetricmountainpasstheorem}. First, by \textnormal{\ref{g3}}, the function $g$ is odd, hence $G$ is even. Therefore
the functional $J$ is even. Moreover, since $G(0)=0$, one has
$J(0)=0$. Next, Proposition \ref{prop:PS-C} shows that $J$ satisfies the Palais--Smale condition
if $p\ge2$, and the Cerami condition if $1<p<2$. Furthermore, Lemma \ref{lem:geometry} yields that condition \textnormal{(ii)} of the symmetric mountain pass theorem \ref{symmetricmountainpasstheorem}. It remains to verify that the third geometric condition holds. Let 
$W\subset W_0^{m,\mcH}(\Om)$ be any finite-dimensional subspace. 
Since all norms are equivalent on $W$, there exists $c_W>0$ such that
\begin{equation}\label{eq:norm-equivalence-W}
\left(\int_{\Om}K(x)|u|^{\gamma q}\,dx\right)^{\frac1{\gamma q}}
\ge c_W \|u\|
\qquad \text{for all } u\in W.
\end{equation}

By \textnormal{\ref{g4}}, for every $A>0$ there exists $C_A>0$ such that
\begin{equation}\label{eq:G-lower}
G(t)\ge A|t|^{\gamma q}-C_A
\qquad \text{for all } t\in\R.
\end{equation}

On the other hand, using \textnormal{\ref{M1}} and \eqref{eq:min-max-WH}, there exists $C>0$ such that
\[
\mcM(\Phi_{\mcH}(\nabla^m u))
\le C(1+\|u\|^{\gamma q})
\qquad \text{for all } \|u\|>1.
\]

Hence, for $u\in W$ with $\|u\|>1$, we obtain
\begin{align*}
J(u)
&=
\mcM(\Phi_{\mcH}(\nabla^m u))
-
\int_{\Om}K(x)G(u)\,dx \\
&\le
C(1+\|u\|^{\gamma q})
-
A\int_{\Om}K(x)|u|^{\gamma q}\,dx
+
C_A \\
&\le
\bigl(C-Ac_W^{\gamma q}\bigr)\|u\|^{\gamma q}+C_1.
\end{align*}

Choosing $A>0$ sufficiently large, we deduce that
\[
J(u)\to -\infty
\qquad \text{as } \|u\|\to\infty,\ u\in W.
\]
Thus, there exists $R>0$ such that
\[
J(u)\le 0
\qquad \text{for all } u\in W \text{ with } \|u\|\ge R.
\]

Therefore, the third geometric condition is satisfied. 
The conclusion follows from the symmetric mountain pass theorem.
\end{proof}

\section{A Mountain Pass Solution}
In this section, we study the  problem \eqref{eq:main-problem} i.e.
\begin{equation}\label{eq:main-problem-sec4}
\begin{cases}
M\!\left(\Phi_{\mcH}(\nabla^m u)\right)\mathcal{L}^{m}_{p,q,a}u
=
h(x,u),
& \text{in } \Omega,\\[0.2cm]
u=\nabla u=\cdots=\nabla^{m-1}u=0,
& \text{on } \partial\Omega.
\end{cases}
\end{equation}
Here $h:\Omega\times\mathbb R\to\mathbb R$ is a continuous function such that
\[
H(x,t):=\int_0^t h(x,s)\,ds.
\]
We impose the following assumptions on the nonlinearity $h$:
\begin{enumerate}[label=\textnormal{(h\arabic*)},leftmargin=2.1cm]

\item\label{h1}
There exists $C>0$ such that
\[
C|h(x,t)|^{(p_m^*)'}
\le
t h(x,t)-\gamma q\,H(x,t)
\]
for a.e. $x\in\Omega$ and all $t\in\mathbb R$.

\item\label{h2}
\[
\lim_{|t|\to\infty}
\frac{h(x,t)}{|t|^{p_m^*-1}}=0
\qquad
\text{uniformly for a.e. }x\in\Omega.
\]

\item\label{h3}
$h$ is odd, i.e.,
\[
h(x,-t)=-h(x,t)\qquad \text{for all }t\in\R~\text{and for a.e. }x\in\Omega.
\]

\item\label{h4}
\[
\lim_{|t|\to\infty}
\frac{H(x,t)}{|t|^{\gamma q}}
=\infty\qquad
\text{uniformly for a.e. }x\in\Omega.
\]
\end{enumerate}

In addition to \textnormal{\ref{M1}}--\textnormal{\ref{M2}}, we assume:

\begin{enumerate}[label=\textnormal{(M\arabic*)},leftmargin=2.1cm]
\setcounter{enumi}{2}
\item\label{M3}
There exists $C_0>0$ such that
\[
\mcM(t)\ge C_0\, t^\gamma
\qquad \text{for all } t\ge 0.
\]
\end{enumerate}

The associated energy functional corresponding to \eqref{eq:main-problem-sec4} is
\begin{equation}\label{eq:J-sec4}
\mathcal{J}(u)
=
\mcM\!\left(\Phi_{\mcH}(\nabla^m u)\right)
-
\int_{\Omega}H(x,u)\,dx,
\qquad
u\in W_0^{m,\mcH}(\Omega).
\end{equation}
\subsection{\texorpdfstring{The $\gamma q$-growth case}{The gamma q-growth case}}
In analogy with the quantity introduced in \cite{Harrabi-Hamdani-Fiscella-2024}, we define
\begin{equation}\label{eq:lambdaH}
\lambda_{\mcH,q}
:=
\inf_{\substack{u\in W_0^{m,\mcH}(\Om)\\ u\neq 0}}
\frac{\mcM\!\left(\Phi_{\mcH}(\nabla^m u)\right)}
{\displaystyle\int_{\Om}|u|^{\gamma q}\,dx}.
\end{equation}

\begin{lemma}\label{lem:lambdaH-positive}
Assume \textnormal{\ref{M3}} holds. Then $\lambda_{\mcH,q}>0$.
\end{lemma}

\begin{proof}
Assume that M satisfies \textnormal{\ref{M3}}. As $1<\gamma q<p_m^*$, by \eqref{embedding} and  \textnormal{\ref{M3}}, we have \(\lambda_{\mcH,q}>0\).
\end{proof}

\begin{remark}\label{rem:lambdaH}
The attainability of $\lambda_{\mcH,q}$ is not required in the sequel. Indeed, by the definition of the infimum, for every $\varepsilon>0$ there exists $u_\varepsilon\in W_0^{m,\mcH}(\Om)\setminus\{0\}$ such that
\[
\mcM\!\left(\Phi_{\mcH}(\nabla^m u_\varepsilon)\right)
\le
(\lambda_{\mcH,q}+\varepsilon)\int_{\Om}|u_\varepsilon|^{\gamma q}\,dx.
\]
Such an approximate minimizer is sufficient for the variational arguments below.
\end{remark}

To derive a mountain pass solution, we impose the following condition at zero and at infinity:

We replace \textnormal{\ref{h4}} by the following assumption:

\begin{enumerate}[label=\textnormal{(h\arabic*)},leftmargin=2.1cm]
\setcounter{enumi}{4}
\item\label{h5}
\[
\limsup_{t\to 0}\frac{H(x,t)}{|t|^{\gamma q}}
<
\lambda_{\mcH,q}
<
\liminf_{|t|\to\infty}\frac{H(x,t)}{|t|^{\gamma q}}
\quad \text{uniformly for a.e. }x\in\Om.
\]
\end{enumerate}

\begin{theorem}\label{thm:mountain-pass}
Assume \textnormal{\ref{M1}} with $t_0=0$, \textnormal{\ref{M2}}, \textnormal{\ref{M3}}, \textnormal{\ref{h1}}, \textnormal{\ref{h2}}, and \textnormal{\ref{h5}} hold. Then problem \eqref{eq:main-problem-sec4} admits a nontrivial weak solution of mountain pass type.
\end{theorem}

\begin{proof}
We verify the assumptions of the standard mountain pass theorem.

First, by similar approach as in Proposition~\ref{prop:PS-C}, the functional $\mathcal{J}$ satisfies the Palais--Smale condition when $p\ge 2$, and the Cerami condition when $1<p<2$.

Next, by combining \textnormal{\ref{h2}} and \textnormal{\ref{h5}}, we can choose $\varepsilon_0>0$ sufficiently small and $C_0>0$ such that
\begin{equation}\label{eq:G-upper-mp}
H(x,t)\le (\lambda_{\mcH,q}-\varepsilon_0)|t|^{\gamma q}+C_0|t|^{p_m^*}
\qquad \text{for all } (x,t)\in \Omega\times \mathbb{R}.
\end{equation}
Hence, by using \ref{M3} and \eqref{eq:G-upper-mp} for every $u\in W_0^{m,\mcH}(\Omega)$, we have
\begin{align}\label{energy}
\mathcal{J}(u)
&=\mcM\bigl(\Phi_{\mcH}(\nabla^m u)\bigr)-\int_\Omega H(x,u)\,dx \\
&\ge \mcM\bigl(\Phi_{\mcH}(\nabla^m u)\bigr)
-(\lambda_{\mcH,q}-\varepsilon_0)\int_\Omega |u|^{\gamma q}\,dx
-C_0\int_\Omega |u|^{p_m^*}\,dx.
\end{align}
Using \eqref{eq:lambdaH}, we have
\begin{align*}
\mathcal{J}(u)
&\ge
\left(1-\frac{\lambda_{\mcH,q}-\varepsilon_0}{\lambda_{\mcH,q}}\right)
\mcM\bigl(\Phi_{\mcH}(\nabla^m u)\bigr)
-C_0\int_\Omega |u|^{p_m^*}\,dx \\
&=
\frac{\varepsilon_0}{\lambda_{\mcH,q}}
\mcM\bigl(\Phi_{\mcH}(\nabla^m u)\bigr)
-C_0\int_\Omega |u|^{p_m^*}\,dx.
\end{align*}
Now, by \textnormal{\ref{M3}},
\[
\mathcal{J}(u)\ge \frac{C\varepsilon_0}{\lambda_{\mcH,q}}
\bigl(\Phi_{\mcH}(\nabla^m u)\bigr)^\gamma
-C_0\int_\Omega |u|^{p_m^*}\,dx.
\]
If $\|u\|=\rho$ with $0<\rho\le 1$, then, by using the estimate \eqref{eq:min-max-WH} for small norms, we have
\[
\Phi_{\mcH}(\nabla^m u)\ge c\,\|u\|^q=c\,\rho^q
\]
for some $c>0$. Moreover, by the Sobolev embedding
\[
W_0^{m,\mcH}(\Omega)\hookrightarrow L^{p_m^*}(\Omega),
\]
there exists $C_1>0$ such that
\[
\int_\Omega |u|^{p_m^*}\,dx\le C_1\|u\|^{p_m^*}=C_1\rho^{p_m^*}.
\]
Hence
\[
\mathcal{J}(u)\ge \frac{Cc\varepsilon_0}{\lambda_{\mcH,q}}\rho^{\gamma q}-C_2\rho^{p_m^*}
=\rho^{\gamma q}\left(\frac{Cc\varepsilon_0}{\lambda_{\mcH,q}}-C_2\rho^{p_m^*-\gamma q}\right)
\]
for some constant $C_2>0$. Since $\gamma q<p_m^*$, we may choose
\[
\rho=\min\left\{1,\left(\frac{Cc\varepsilon_0}{2C_2\lambda_{\mcH,q}}\right)^{\frac{1}{p_m^*-\gamma q}}\right\}
\]
and
\[
\alpha=\frac{Cc\varepsilon_0}{2\lambda_{\mcH,q}}\rho^{\gamma q}>0,
\]
so that
\[
\mathcal{J}(u)\ge \alpha
\qquad \text{for all } \|u\|=\rho.
\]

On the other hand, by \textnormal{\ref{h5}}, and Remark \ref{rem:lambdaH}, for $\varepsilon_0>0$ sufficiently small, there exists a constant $C_3>0$ and a function
$\varphi_\epsilon\in W_0^{m,\mcH}(\Omega)\setminus\{0\}$ such that
\begin{equation}\label{eq:G-lower-mp}
H(x,t)\ge (\lambda_{\mcH,q}+2\varepsilon_0)|t|^{\gamma q}-C_3
\qquad \text{for all } (x,t)\in \Omega\times \mathbb{R},
\end{equation}
and
\begin{equation}\label{eq:phi-choice-mp}
\lambda_{\mcH,q}\int_\Omega |\varphi_\epsilon|^{\gamma q}\,dx
\le
\mcM\bigl(\Phi_{\mcH}(\nabla^m\varphi_\epsilon)\bigr)
\le
(\lambda_{\mcH,q}+\varepsilon_0)\int_\Omega |\varphi_\epsilon|^{\gamma q}\,dx.
\end{equation}

Moreover, assumption \textnormal{\ref{M1}} with $t_0=0$ implies that, for every $s_1>0$,
\[
\frac{\mcM(s)}{s^\gamma}\le \frac{\mcM(s_1)}{s_1^\gamma}
\qquad \text{for all } s\ge s_1.
\]
Now let $v=t\varphi_\epsilon$ with $t\ge 1$. Since $\nabla^m(t\varphi_\epsilon)=t\nabla^m\varphi_\epsilon,$
we have
\begin{align*}
\Phi_{\mcH}(\nabla^m(t\varphi_\epsilon))
\leq
t^q\Phi_{\mcH}(\nabla^m\varphi_\epsilon).
\end{align*}
Applying the above monotonicity with $s=\Phi_{\mcH}(\nabla^m(t\varphi_\epsilon))$ and
$s_1=\Phi_{\mcH}(\nabla^m\varphi_\epsilon)$, we obtain
\begin{align}\label{mestimate}
\mcM\bigl(\Phi_{\mcH}(\nabla^m(t\varphi_\epsilon))\bigr)
\le
t^{\gamma q}\mcM\bigl(\Phi_{\mcH}(\nabla^m\varphi_\epsilon)\bigr),
\qquad t\ge 1.
\end{align}
Therefore, by \eqref{energy}, \eqref{eq:G-lower-mp}, \eqref{eq:phi-choice-mp}, and \eqref{mestimate}, we have for $t\geq 1$
\begin{align*}
\mathcal{J}(t\varphi_\epsilon)
&=
\mcM\bigl(\Phi_{\mcH}(\nabla^m(t\varphi_\epsilon))\bigr)-\int_\Omega H(x,t\varphi_\epsilon)\,dx \\
&\le
t^{\gamma q}\mcM\bigl(\Phi_{\mcH}(\nabla^m\varphi_\epsilon)\bigr)
-(\lambda_{\mcH,q}+2\varepsilon_0)t^{\gamma q}\int_\Omega |\varphi_\epsilon|^{\gamma q}\,dx
+C_3|\Omega|\\
&\le
\left[(\lambda_{\mcH,q}+\varepsilon_0)-(\lambda_{\mcH,q}+2\varepsilon_0)\right]
t^{\gamma q}\int_\Omega |\varphi_\epsilon|^{\gamma q}\,dx
+C_3|\Omega| \\
&=
-\varepsilon_0 t^{\gamma q}\int_\Omega |\varphi_\epsilon|^{\gamma q}\,dx
+C_3|\Omega|.
\end{align*}
Hence,
\[
\mathcal{J}(t\varphi_\epsilon)\to -\infty
\qquad \text{as } t\to \infty.
\]
Choosing $t$ sufficiently large, we get some $e=t\varphi_\epsilon$ such that
\[
\|e\|>\rho
\qquad \text{and} \qquad
\mathcal{J}(e)<0.
\]

Thus, $\mathcal{J}$ has the mountain pass geometry. This completes the proof.
\end{proof}

\subsection{\texorpdfstring{The $\gamma p$-sublinear case}{The gamma p-sublinear case}}\label{subsec:gamma-p-sublinear}

In this subsection, we study the sublinear case governed by the exponent $\gamma p$. Assume throughout this subsection that
$
1<\gamma p<p_m^*.
$
This exponent $\gamma p$ is natural in the double phase setting near the small-amplitude regime,
since for $0<\tau\le 1$ one has
\[
\Phi_{\mcH}(\nabla^m(\tau u))
=
\int_\Omega
\left(
\frac{\tau^p}{p}|\nabla^m u|^p
+
a(x)\frac{\tau^q}{q}|\nabla^m u|^q
\right)\,dx
\le
\tau^p\Phi_{\mcH}(\nabla^m u).
\]
This is different from the large-amplitude regime, where the exponent $\gamma q$
naturally appears. Therefore, in the $\gamma p$-sublinear case we introduce the
following threshold constant:
\begin{equation}\label{eq:lambdaHp}
\lambda_{\mcH,p}
:=
\inf_{\substack{u\in W_0^{m,\mcH}(\Omega)\\ u\neq 0}}
\frac{\mcM\!\left(\Phi_{\mcH}(\nabla^m u)\right)}
{\displaystyle\int_\Omega |u|^{\gamma p}\,dx}.
\end{equation}

\begin{lemma}\label{lem:lambdaHp-positive}
Assume \textnormal{\ref{M3}} holds. Then $\lambda_{\mcH,p}>0$.
\end{lemma}

\begin{proof}
Assume that M satisfies \textnormal{\ref{M3}}. As $1<\gamma p<p_m^*$, by \eqref{embedding} and  \textnormal{\ref{M3}}, we have \(\lambda_{\mcH,p}>0\).
\end{proof}

Throughout this subsection, we assume the following subcritical growth condition:
there exist $r\in(1,p_m^*)$ and $C_r>0$ such that
\begin{equation}\label{eq:g-growth-sec42}
|h(x,t)|
\le
C_r\left(1+|t|^{r-1}\right)
\qquad
\text{for a.e. }x\in\Omega \text{ and all }t\in\mathbb R.
\end{equation}

We impose the following $\gamma p$-sublinear condition at infinity:
\begin{enumerate}[label=\textnormal{(h\arabic*)},leftmargin=2.1cm]
\setcounter{enumi}{5}
\item\label{h6}
\[
\limsup_{|t|\to\infty}
\frac{H(x,t)}{|t|^{\gamma p}}
<
\lambda_{\mcH,p}
\qquad \text{uniformly for a.e. }x\in\Omega.
\]
\end{enumerate}

\begin{proposition}\label{prop:sublinear}
Assume \textnormal{\ref{M3}}, \eqref{eq:g-growth-sec42}, and
\textnormal{\ref{h6}} hold. Then
\begin{enumerate}[label=\textnormal{(\roman*)}]
\item $\mathcal{J}(u)\to\infty$ as $\|u\|\to\infty$;
\item $\mathcal{J}$ satisfies the Palais--Smale condition.
\end{enumerate}
\end{proposition}

\begin{proof}
Fix $\varepsilon_0\in(0,\lambda_{\mcH,p})$. By \textnormal{\ref{h6}}, there exists
$C_0>0$ such that
\begin{equation}\label{eq:G-sublinear-est}
H(x,t)
\le
(\lambda_{\mcH,p}-\varepsilon_0)|t|^{\gamma p}+C_0
\qquad
\text{for a.e. }x\in\Omega \text{ and all }t\in\mathbb R.
\end{equation}
Therefore, for every $u\in W_0^{m,\mcH}(\Omega)$, we have
\begin{align}
\mathcal{J}(u)
&=
\mcM\!\left(\Phi_{\mcH}(\nabla^m u)\right)
-
\int_\Omega H(x,u)\,dx \nonumber\\
&\ge
\mcM\!\left(\Phi_{\mcH}(\nabla^m u)\right)
-
(\lambda_{\mcH,p}-\varepsilon_0)
\int_\Omega |u|^{\gamma p}\,dx
-
C_0|\Omega|.
\label{eq:J-sublinear-lower}
\end{align}
By using \eqref{eq:lambdaHp}, we obtain
\[
\mathcal{J}(u)
\ge
\frac{\varepsilon_0}{\lambda_{\mcH,p}}
\mcM\!\left(\Phi_{\mcH}(\nabla^m u)\right)
-
C_0|\Omega|.
\]
Using \textnormal{\ref{M3}}, we get
\[
\mathcal{J}(u)
\ge
C_1\left(\Phi_{\mcH}(\nabla^m u)\right)^\gamma
-
C_0|\Omega|
\]
for some constant $C_1>0$. Since
\[
\Phi_{\mcH}(\nabla^m u)\to\infty
\qquad
\text{as }\|u\|\to\infty,
\]
we conclude that
\[
\mathcal{J}(u)\to\infty
\qquad
\text{as }\|u\|\to\infty.
\]
This proves \textnormal{(i)}.

We now prove \textnormal{(ii)}. Let $\{u_n\}$ be a Palais--Smale sequence for
$\mathcal{J}$. Since $\mathcal{J}$ is coercive by \textnormal{(i)}, the sequence
$\{u_n\}$ is bounded in $W_0^{m,\mcH}(\Omega)$. Moreover, by the subcritical
growth condition \eqref{eq:g-growth-sec42}, the compactness argument used in the
second step of the proof of Proposition~\ref{prop:PS-C} applies to bounded
Palais--Smale sequences. Hence $\{u_n\}$ admits a strongly convergent subsequence
in $W_0^{m,\mcH}(\Omega)$. Therefore, $\mathcal{J}$ satisfies the Palais--Smale
condition.
\end{proof}

We now impose the following strong lower growth condition near the origin:
\begin{enumerate}[label=\textnormal{(h\arabic*)},leftmargin=2.1cm]
\setcounter{enumi}{6}
\item\label{h7}
\[
\liminf_{|t|\to0}
\frac{H(x,t)}{|t|^{\gamma p}}
=
+\infty
\qquad
\text{uniformly for a.e. }x\in\Omega.
\]
\end{enumerate}

We also assume the following upper estimate on the Kirchhoff potential:
\begin{enumerate}[label=\textnormal{(M\arabic*)},leftmargin=2.1cm]
\setcounter{enumi}{3}
\item\label{M4}
There exists $\beta>0$ such that
\[
\mcM(t)\le \beta t^\gamma
\qquad
\text{for all }t\ge0.
\]
\end{enumerate}

\begin{theorem}\label{thm:4.4}
Assume that \eqref{eq:g-growth-sec42}, \textnormal{\ref{M2}}, \textnormal{\ref{M3}}, \textnormal{\ref{M4}}, \textnormal{\ref{h6}}, and \textnormal{\ref{h7}} hold. 
Suppose moreover that $h(x,\cdot)$ is odd for a.e. $x\in\Omega$. Then the functional $\mathcal{J}$
admits infinitely many distinct pairs of nontrivial critical points.
\end{theorem}

\begin{proof}
We will verify the assumptions of Theorem~\ref{genustheorem}.

Since $h(x,\cdot)$ is odd for a.e. $x\in\Omega$, its primitive $H(x,\cdot)$ is even. Hence
the functional $\mathcal{J}$ is even and clearly $\mathcal{J}(0)=0$. Moreover, by Proposition~\ref{prop:sublinear},
the functional $\mathcal{J}$ satisfies the Palais--Smale condition.

Let $n\in\mathbb{N}^*$. We first show that $\Gamma_n\neq\emptyset$. Choose functions
\[
\phi_1,\phi_2,\dots,\phi_n\in C_c^\infty(\Omega)
\]
such that
\[
\|\phi_i\|_{L^2(\Omega)}=1
\qquad \text{and} \qquad
\operatorname{supp}(\phi_i)\cap \operatorname{supp}(\phi_j)=\emptyset
\quad \text{whenever } i\neq j, ~\text{and}~1\leq i,j\leq n.
\]
Set
\[
E_n:=\operatorname{span}\{\phi_1,\phi_2,\dots,\phi_n\}\subset W_0^{m,\mcH}(\Omega)\cap L^2(\Omega).
\]
For \(0<\sigma<1\), define
\[
S_n^\sigma:=\{u\in E_n:\|u\|_{L^2(\Omega)}=\sigma\}.
\]
Since every \(u\in E_n\) can be uniquely written as
\[
u=\sum_{i=1}^n \lambda_i \phi_i,
\]
the map
\[
\theta:S_n^\sigma\to {S}^{n-1},
\qquad
\theta(u)=\left(\frac{\lambda_1}{\sigma},\frac{\lambda_2}{\sigma},\dots,\frac{\lambda_n}{\sigma}\right),
\]
is an odd homeomorphism, where $S^{n-1}$ is the sphere of dimension ${n-1}$. Therefore, \(\gamma(S_n^\sigma)=n\); see \cite{Rabinowitz-1986}. Consequently,
\(S_n^\sigma\in\Gamma_n\). In particular, \(\Gamma_n\neq\emptyset\). We now prove that \(c_n<0\), where
\[
c_n:=\inf_{A\in\Gamma_n}\sup_{u\in A}\mathcal{J}(u).
\]
Since Proposition~\ref{prop:sublinear} yields that \(\mathcal{J}\) is coercive, the functional \(\mathcal{J}\) is bounded
from below, and hence
$
-\infty<c_n.
$

By \textnormal{\ref{h7}}, for every \(A>0\), there exists \(t_A>0\) such that
\begin{equation}
H(x,t)\ge A|t|^{\gamma p},
\qquad
\text{for a.e. }x\in\Omega,\ |t|\le t_A .
\label{eq:G-small}
\end{equation}
Set
\[
M_n:=\max\left\{\|\phi_i\|_{L^\infty(\Omega)}:1\le i\le n\right\}.
\]
Then, for
\[
\sigma:=\min\left\{\frac12,\frac{t_A}{2nM_n}\right\},
\]
we have
\begin{equation}
\|u\|_{L^\infty(\Omega)}\le \frac{t_A}{2}
\qquad
\text{for all }u\in S_n^\sigma.
\end{equation}
Therefore, by \eqref{eq:G-small},
\begin{equation}\label{eq:G-est-44}
    \int_\Omega H(x,u)\,dx
\ge
A\int_\Omega |u|^{\gamma p}\,dx,
\qquad
u\in S_n^\sigma .
\end{equation}

Fix \(n\in\mathbb{N}^*\). Since all norms are equivalent on the finite-dimensional space \(E_n\), there exists
\(C_n>0\) such that
\begin{equation}
\frac{1}{C_n}\|u\|
\le
\|u\|_{L^2(\Omega)}
\le
C_n\|u\|_{L^{\gamma p}(\Omega)}
\qquad
\text{for all }u\in E_n.
\label{eq:En-equivalence}
\end{equation}
By \textnormal{\ref{M4}} and \eqref{eq:En-equivalence}, we get
\begin{equation}
\mcM\bigl(\Phi_{\mcH}(\nabla^m u)\bigr)
\le
\widetilde C_n\|u\|_{L^2(\Omega)}^{\gamma p}
\qquad
\text{for all }u\in S_n^\sigma
\label{eq:M-upper-44}
\end{equation}
for some \(\widetilde C_n>0\). Moreover, from \eqref{eq:En-equivalence},
\begin{equation}
\int_\Omega |u|^{\gamma p}\,dx
=
\|u\|_{L^{\gamma p}(\Omega)}^{\gamma p}
\ge
\frac{1}{C_n^{\gamma p}}\|u\|_{L^2(\Omega)}^{\gamma p}.
\label{eq:Lgp-lower-44}
\end{equation}
Combining \eqref{eq:G-est-44}--\eqref{eq:Lgp-lower-44}, we obtain
\[
\mathcal{J}(u)
\le
\left(
\widetilde C_n-\frac{A}{C_n^{\gamma p}}
\right)
\|u\|_{L^2(\Omega)}^{\gamma p}
\qquad
\text{for all }u\in S_n^\sigma.
\]
Choosing \(A>\widetilde C_n C_n^{\gamma p}\), we deduce that
\[
\sup_{u\in S_n^\sigma}\mathcal{J}(u)<0.
\]
Since \(S_n^\sigma\in\Gamma_n\), it follows that
\[
c_n\le \sup_{u\in S_n^\sigma}\mathcal{J}(u)<0.
\]
Thus
\[
-\infty<c_n<0.
\]
By Theorem~\ref{genustheorem}, \(c_n\) is a critical value of \(\mathcal{J}\).
If the sequence $\{c_n\}$ contains infinitely many distinct values, then
$\mathcal{J}$ admits infinitely many nontrivial critical points. Otherwise,
some critical value is repeated infinitely often, and Theorem~\ref{genustheorem}
\textnormal{(ii)} implies that the corresponding critical set has arbitrarily
large genus. Hence it contains infinitely many distinct critical points.
Therefore, $\mathcal{J}$ admits infinitely many nontrivial critical points.

\end{proof}

We finally consider a weaker lower condition near the origin. For this purpose, we assume
the following scaling condition on the Kirchhoff potential:
\begin{enumerate}[label=\textnormal{(M\arabic*)},leftmargin=2.1cm]
\setcounter{enumi}{4}
\item\label{M5}
\[
\mcM(\theta s)\le \theta^\gamma \mcM(s)
\qquad
\text{for all }0<\theta\le1
\text{ and all }s\ge0.
\]
\end{enumerate}

We replace \textnormal{\ref{h7}} by the following weaker condition:
\begin{enumerate}[label=\textnormal{(h\arabic*)},leftmargin=2.1cm]
\setcounter{enumi}{7}
\item\label{h8}
\[
\lambda_{\mcH,p}
<
\liminf_{|t|\to0}
\frac{H(x,t)}{|t|^{\gamma p}}
\qquad
\text{uniformly for a.e. }x\in\Omega.
\]
\end{enumerate}

Under this weaker assumption, one cannot in general apply the Krasnosel'skii genus
argument used in Theorem~\ref{thm:4.4}. Nevertheless, combining the coercivity and
compactness established in Proposition~\ref{prop:sublinear} with a suitable
test-function argument near the origin, one still obtains a nontrivial critical
point of \(\mathcal{J}\) at a negative energy level. This gives the following existence result.

\begin{theorem}\label{thm:4.5}
Assume that \textnormal{\ref{M2}}, \textnormal{\ref{M3}}, \textnormal{\ref{M5}},
\eqref{eq:g-growth-sec42}, \textnormal{\ref{h6}}, and \textnormal{\ref{h8}} hold.
Then the functional $\mathcal{J}$ is bounded from below and
\[
c:=\inf\{\mathcal{J}(u):u\in W_0^{m,\mcH}(\Omega)\}<0
\]
is a critical value of $\mathcal{J}$. Consequently, problem \eqref{eq:main-problem-sec4}
admits a nontrivial weak solution.
\end{theorem}

\begin{proof}
By Proposition \ref{prop:sublinear}, the functional $\mathcal{J}$ is coercive, bounded from below,
and satisfies the Palais--Smale condition. Hence
\[
c:=\inf\{\mathcal{J}(u):u\in W_0^{m,\mcH}(\Omega)\}
\]
is well defined. It remains to prove that $c<0$.
By \textnormal{\ref{h8}}, we may choose $\varepsilon_0>0$ and $t_0>0$ such that
\begin{equation}\label{eq:G-small-thm45}
H(x,t)
\ge
(\lambda_{\mcH,p}+2\varepsilon_0)|t|^{\gamma p}
\qquad
\text{for a.e. }x\in\Omega
\text{ and all }|t|\le t_0.
\end{equation}
By the definition of $\lambda_{\mcH,p}$ and the density of $C_c^\infty(\Omega)$ in
$W_0^{m,\mcH}(\Omega)$, there exists
$\phi\in C_c^\infty(\Omega)\setminus\{0\}$ such that
\begin{equation}\label{eq:phi-thm45}
\mcM\!\left(\Phi_{\mcH}(\nabla^m\phi)\right)
\le
(\lambda_{\mcH,p}+\varepsilon_0)
\int_\Omega |\phi|^{\gamma p}\,dx.
\end{equation}
Set
\[
\tau
:=
\min\left\{
\frac{t_0}{\|\phi\|_{L^\infty}},
\frac12
\right\}.
\]
Then $0<\tau\le1$ and $|\tau\phi(x)|\le t_0$ for a.e. $x\in\Omega$. Therefore, by
\eqref{eq:G-small-thm45},
\begin{equation}\label{eq:G-small-phi-thm45}
\int_\Omega H(x,\tau\phi)\,dx
\ge
(\lambda_{\mcH,p}+2\varepsilon_0)
\tau^{\gamma p}
\int_\Omega |\phi|^{\gamma p}\,dx.
\end{equation}
On the other hand, since $0<\tau\le1$ and $p<q$, we have
\begin{align}
\Phi_{\mcH}(\nabla^m(\tau\phi))
&=
\int_\Omega
\left(
\frac{\tau^p}{p}|\nabla^m\phi|^p
+
a(x)\frac{\tau^q}{q}|\nabla^m\phi|^q
\right)\,dx \nonumber\\
&\le
\tau^p
\int_\Omega
\left(
\frac{1}{p}|\nabla^m\phi|^p
+
a(x)\frac{1}{q}|\nabla^m\phi|^q
\right)\,dx \nonumber\\
&=
\tau^p\Phi_{\mcH}(\nabla^m\phi).
\label{eq:Phi-small-thm45}
\end{align}
Using the monotonicity of $\mcM$, \eqref{eq:Phi-small-thm45}, and
\textnormal{\ref{M5}}, we get
\begin{equation}\label{eq:M-small-thm45}
\mcM\!\left(\Phi_{\mcH}(\nabla^m(\tau\phi))\right)
\le
\mcM\!\left(\tau^p\Phi_{\mcH}(\nabla^m\phi)\right)
\le
\tau^{\gamma p}
\mcM\!\left(\Phi_{\mcH}(\nabla^m\phi)\right).
\end{equation}
Combining \eqref{eq:G-small-phi-thm45}, \eqref{eq:M-small-thm45}, and
\eqref{eq:phi-thm45}, we obtain
\begin{align*}
\mathcal{J}(\tau\phi)
&=
\mcM\!\left(\Phi_{\mcH}(\nabla^m(\tau\phi))\right)
-
\int_\Omega H(x,\tau\phi)\,dx \\
&\le
\tau^{\gamma p}
\mcM\!\left(\Phi_{\mcH}(\nabla^m\phi)\right)
-
(\lambda_{\mcH,p}+2\varepsilon_0)
\tau^{\gamma p}
\int_\Omega |\phi|^{\gamma p}\,dx \\
&\le
(\lambda_{\mcH,p}+\varepsilon_0)
\tau^{\gamma p}
\int_\Omega |\phi|^{\gamma p}\,dx
-
(\lambda_{\mcH,p}+2\varepsilon_0)
\tau^{\gamma p}
\int_\Omega |\phi|^{\gamma p}\,dx \\
&=
-\varepsilon_0
\tau^{\gamma p}
\int_\Omega |\phi|^{\gamma p}\,dx
<0.
\end{align*}
Therefore,
\[
c\le \mathcal{J}(\tau\phi)<0.
\]
Hence $c<0$. Since $\mathcal{J}$ is bounded from below, Ekeland's variational principle gives a minimizing
sequence $\{u_n\}\subset W_0^{m,\mcH}(\Omega)$ such that
\[
\mathcal{J}(u_n)\to c
\qquad
\text{and}
\qquad
\mathcal{J}'(u_n)\to0
\quad
\text{in }\bigl(W_0^{m,\mcH}(\Omega)\bigr)^*.
\]
By the Palais-Smale condition, up to a subsequence,
\[
u_n\to u
\quad
\text{strongly in }W_0^{m,\mcH}(\Omega)
\]
for some $u\in W_0^{m,\mcH}(\Omega)$. Passing to the limit, we obtain
\[
\mathcal{J}(u)=c
\qquad
\text{and}
\qquad
\mathcal{J}'(u)=0.
\]
Thus $c$ is a critical value of $\mathcal{J}$. Since $c<0=\mathcal{J}(0)$, the corresponding critical point
is nontrivial. Therefore, problem \eqref{eq:main-problem-sec4} admits a nontrivial weak
solution.
\end{proof}

\section*{Funding Declaration}
The authors received no specific funding for this work

\end{document}